\documentclass{article}
\usepackage{mathrsfs, amsmath, amsthm, amssymb, mathtools, thm-restate}% packages for math
\usepackage{graphicx, xcolor, tikz}% packages for figures
\usetikzlibrary{arrows.meta, intersections, calc, shapes}
\usepackage{caption, subcaption}
\usepackage{array}
\usepackage{cases}
\makeatletter
\def\th@plain{%
  \upshape %\itshape % body font
}
\makeatother

\makeatletter
\renewenvironment{proof}[1][\proofname]{\par
  \pushQED{\qed}%
  \normalfont \topsep6\p@\@plus6\p@\relax
  \trivlist
  \item[\hskip\labelsep
        \bfseries
    #1\@addpunct{.}]\ignorespaces
}{%
  \popQED\endtrivlist\@endpefalse
}
\makeatother

\newtheorem{theorem}{Theorem}[section]
\newtheorem{lemma}[theorem]{Lemma}
\newtheorem{corollary}[theorem]{Corollary}
\newtheorem{proposition}{Proposition}
\newtheorem{conjecture}{Conjecture}
\newtheorem{case}{Case}
\newtheorem{subcase}{Subcase}[case]
\theoremstyle{definition}
\newtheorem{definition}{Definition}
\newtheorem{remark}{Remark}

\usepackage[left=25mm, top=25mm, bottom=25mm, right=25mm]{geometry}
\usepackage[T1]{fontenc} % https://tex.stackexchange.com/questions/664/why-should-i-use-usepackaget1fontenc
\usepackage{newtxtext}
\usepackage[varvw, varg]{newtxmath}
\usepackage{bm} % the package "bm" and "newtxmath" are conflict in MacTex 2020, so we must use "bm" later
\usepackage{paralist, tablists}% put these packages ahead of enumitem
\usepackage[inline]{enumitem}
\usepackage[square, numbers, sort&compress]{natbib}

\usepackage[pdftex,%
  bookmarks=true,%
  bookmarksnumbered=true, %% 把章节的标号放入书签中，缺省为false
  bookmarksopen=true, %% 打开书签树
  plainpages=false,%
  pdfpagelabels,%
  colorlinks=true, %% 彩色文字链接打开，去掉方框
  linkcolor=blue, %% 目录链接的颜色为蓝色
  citecolor=blue,%
  anchorcolor=green,
  urlcolor= blue,
  breaklinks=true,
  hyperindex=true]{hyperref}%% 不能和cite宏包同时使用

\newcommand{\etal}{et~al.\ }
\newcommand{\ie}{i.e.,\ }
\usepackage{fontawesome}

\begin{document}
\title{Cover and variable degeneracy}
\author{Fangyao Lu { }\quad Qianqian Wang { }\quad Tao Wang\footnote{\tt Corresponding
author: wangtao@henu.edu.cn; https://orcid.org/0000-0001-9732-1617}\\
{\small Center for Applied Mathematics}\\
{\small Henan University, Kaifeng, 475004, P. R. China}}
\date{}
\maketitle
\begin{abstract}
Let $f$ be a nonnegative integer valued function on the vertex set of a graph. A graph is \textbf{strictly $f$-degenerate} if each nonempty subgraph $\Gamma$ has a vertex $v$ such that $\mathrm{deg}_{\Gamma}(v) < f(v)$. In this paper, we define a new concept, strictly $f$-degenerate transversal, which generalizes list coloring, signed coloring, DP-coloring, $L$-forested-coloring, and $(f_{1}, f_{2}, \dots, f_{s})$-partition. A \textbf{cover} of a graph $G$ is a graph $H$ with vertex set $V(H) = \bigcup_{v \in V(G)} X_{v}$, where $X_{v} = \{(v, 1), (v, 2), \dots, (v, s)\}$; the edge set $\mathscr{M} = \bigcup_{uv \in E(G)}\mathscr{M}_{uv}$, where $\mathscr{M}_{uv}$ is a matching between $X_{u}$ and $X_{v}$. A vertex set $R \subseteq V(H)$ is a \textbf{transversal} of $H$ if $|R \cap X_{v}| = 1$ for each $v \in V(G)$. A transversal $R$ is a \textbf{strictly $f$-degenerate transversal} if $H[R]$ is strictly $f$-degenerate. The main result of this paper is a degree type result, which generalizes Brooks' theorem, Gallai's theorem, degree-choosable result, signed degree-colorable result, and DP-degree-colorable result. We also give some structural results on critical graphs with respect to strictly $f$-degenerate transversal. Using these results, we can uniformly prove many new and known results. In the final section, we pose some open problems. 
\end{abstract}

{\bf Keywords}: Strictly f-degenerate; List vertex arboricity; DP-coloring; Signed coloring; List-coloring

{\bf MSC2020}: 05C15

\section{Introduction}\label{sec1}
All graphs considered in this paper are simple, finite and undirected. Let $\mathbb{Z}^{*}$ stand for the set of nonnegative integers. Let $s$ be a positive integer and let $[s]$ stand for the set $\{1, 2, \dots, s\}$. 

For a graph $G$, a \textbf{list-assignment} $L$ assigns to each vertex $v$ a set $L(v)$ of colors for $v$. An \textbf{$L$-coloring} of $G$ is a proper coloring $\phi$ of $G$ such that $\phi(v) \in L(v)$ for all $v \in V(G)$. A \textbf{list $k$-assignment} $L$ is a list-assignment such that $|L(v)| \geq k$ for all $v \in V(G)$. A graph $G$ is \textbf{$k$-choosable} or \textbf{list $k$-colorable} if it has an $L$-coloring for any list $k$-assignment $L$. The \textbf{list chromatic number} or \textbf{choice number} $\chi_{\ell}(G)$ is the least integer $k$ such that $G$ is $k$-choosable. 

A graph is \textbf{$k$-degenerate} if each nonempty subgraph has a vertex of degree at most $k$ in this subgraph. The \textbf{degeneracy} of $G$ is the least integer $k$ such that it is $k$-degenerate. A graph is \textbf{strictly $k$-degenerate} if each nonempty subgraph has a vertex of degree less than $k$ in this subgraph. A graph is strictly $1$-degenerate if and only if it is edgeless, and a graph is strictly $2$-degenerate if and only if it is a forest. Note that a graph is strictly $k$-degenerate if and only if it is $(k-1)$-degenerate. The \textbf{coloring number} of $G$ is the least integer $k$ such that $G$ is strictly $k$-degenerate. Note that the coloring number equals the degeneracy plus one. Borodin, Kostochka and Toft \cite{MR1743629} introduced a more flexible concept, strictly $f$-degenerate, to unify some problems. Let $f$ be a function from $V(G)$ to $\mathbb{Z}^{*}$. A graph $G$ is \textbf{strictly $f$-degenerate} if each nonempty subgraph $\Gamma$ of $G$ has a vertex $v$ such that $\deg_{\Gamma}(v) < f(v)$. In other words, the vertices of $G$ can be ordered as $v_{1}, v_{2}, \dots, v_{n}$ such that each vertex $v_{i}$ has less than $f(v_{i})$ neighbors with larger subscripts. This order is called an \textbf{$f$-removing order}. 

The \textbf{vertex arboricity} of a graph $G$ is the minimum number of subsets (color classes) in which $V(G)$ can be partitioned so that each subset induces a forest. The \textbf{linear vertex arboricity} of a graph $G$ is the minimum number of subsets (color classes) in which $V(G)$ can be partitioned so that each subset induces a forest with maximum degree at most two. An \textbf{$L$-forested-coloring} of $G$ for a list assignment $L$ is a coloring $\phi$ (not necessarily proper) such that $\phi(v) \in L(v)$ for each $v \in V(G)$ and each color class induces a forest. The \textbf{list vertex arboricity} of a graph $G$ is the least integer $k$ such that $G$ has an $L$-forested-coloring for any list $k$-assignment $L$. The list vertex arboricity of a graph is the list version of vertex arboricity. The \textbf{linear list vertex arboricity} can be similarly defined. 

Let $f_{1}, f_{2}, \dots, f_{s}$ be functions from $V(G)$ to $\mathbb{Z}^{*}$. A partition $(V_{1}, V_{2}, \dots, V_{s})$ of $V(G)$ is an \textbf{$(f_{1}, f_{2}, \dots, f_{s})$-partition} if $G[V_{i}]$ is strictly $f_{i}$-degenerate for each $i \in [s]$. A graph $G$ is \textbf{$(f_{1}, f_{2}, \dots, f_{s})$-partitionable} if $V(G)$ has a $(f_{1}, f_{2}, \dots, f_{s})$-partition. Note that if $f_{i}(v) = 0$ then $v$ cannot be chosen in $V_{i}$ since it never has a negative degree. Partitioning $V(G)$ into induced subgraphs of variable degeneracy in which $f_{i}(v) \in \{0, 1\}$ for all $v$ and $i$, corresponds to list coloring of $G$ with the list $L(v) = \{\,i \mid f_{i}(v) = 1\,\}$. Then the $(f_{1}, f_{2}, \dots, f_{s})$-partitionability is a generalization of choosability. Similarly, partitioning $V(G)$ into induced subgraphs of variable degeneracy in which $f_{i}(v) \in \{0, 2\}$ for all $v$ and $i$, corresponds to $L$-forested-coloring of $G$ with the list $L(v) = \{\,i \mid f_{i}(v) = 2\,\}$. Then the $(f_{1}, f_{2}, \dots, f_{s})$-partition is a generalization of $L$-forested-coloring. 

A \textbf{signed graph} is a graph in which each edge is labeled with a sign $\sigma$ that is either $+1$ or $-1$. Let $\mathcal{C}_{k} \coloneqq \{\pm 1, \pm 2, \dots, \pm s\}$ if $k = 2s$, and $\mathcal{C}_{k} \coloneqq \{0, \pm 1, \pm 2, \dots, \pm s\}$ if $k = 2s+1$. A \textbf{signed $k$-coloring} of a signed graph $(G, \sigma)$ is a mapping $\phi$ from $V(G)$ to $\mathcal{C}_{k}$ such that $\phi(u) \neq \sigma(uv) \times \phi(v)$ for each edge $uv$ in $G$. The \textbf{signed chromatic number $\chi^{\pm}(G)$} of a graph $G$ is the least integer $k$ such that for any signed graph $(G, \sigma)$, it has a signed $k$-coloring. The concept of signed graphs is due to Harary \cite{MR0067468}, but the current concept of signed chromatic number is due to M\'{a}\v{c}ajov\'{a}, Raspaud and \v{S}koviera \cite{MR3484719}. For more results on signed coloring, we refer the reader to \cite{MR3612419, MR3425977, MR3477382}. 

A \textbf{cover} of a graph $G$ is a graph $H$ with vertex set $V(H) = \bigcup_{v \in V(G)} X_{v}$, where $X_{v} = \{(v, 1), (v, 2), \dots, (v, s)\}$, and edge set $\mathscr{M} = \bigcup_{uv \in E(G)}\mathscr{M}_{uv}$, where $\mathscr{M}_{uv}$ is a matching between $X_{u}$ and $X_{v}$. Note that $X_{v}$ is an independent set in $H$ and $\mathscr{M}_{uv}$ may be an empty set. For convenience, this definition is slightly different from Bernshteyn and Kostochka's, but consistent with Schweser's \cite{MR4020554}. A vertex subset $R \subseteq V(H)$ is a \textbf{transversal} of $H$ if $|R \cap X_{v}| = 1$ for each $v \in V(G)$. 

Let $H$ be a cover of $G$ and $f$ be a function from $V(H)$ to $\mathbb{Z}^{*}$, we call the pair $(H, f)$ a \textbf{valued cover} of $G$. We say that $f$ is \textbf{semi-constant} if it is a constant on each component of $H$. Let $S$ be a subset of $V(G)$, we use $H_{S}$ to denote the induced subgraph $H[\bigcup_{v \in S} X_{v}]$. A transversal $R$ is a \textbf{strictly $f$-degenerate transversal} if $H[R]$ is strictly $f$-degenerate. We say that the vertex $v$ in $G$ \textbf{is colored with $p$} if $(v, p)$ is chosen in a strictly $f$-degenerate transversal of $H$. 

Let $H$ be a cover of $G$ and $f$ be a function from $V(H)$ to $\{0, 1\}$. An \textbf{independent transversal}, or \textbf{DP-coloring}, of $H$ is a strictly $f$-degenerate transversal of $H$. Note that a DP-coloring is a special independent set in $H$. DP-coloring, also known as \textbf{correspondence coloring}, was introduced by Dvo\v{r}\'{a}k and Postle \cite{MR3758240} when they used it to solve a longstanding conjecture by Borodin \cite{MR3004485}. More results and discussions on DP-coloring, see \cite{MR3679840, MR3518419, MR3889157, MR4043754, MR3957361, MR3948125}. 

The \textbf{DP-chromatic number $\chi_{\mathrm{DP}}(G)$} of $G$ is the least integer $k$ such that $(H, f)$ has a DP-coloring whenever $H$ is a cover of $G$ and $f$ is a function from $V(H)$ to $\{0, 1\}$ with $f(v, 1) + f(v, 2) + \dots + f(v, s) \geq k$ for each $v \in V(G)$. A graph $G$ is \textbf{DP-$k$-colorable} if its DP-chromatic number is at most $k$. 

The \textbf{Cartesian product} $G \Box H$ of graphs $G$ and $H$ is the simple graph with vertex set $V(G) \times V(H)$, in which two vertices $(u, u')$ and $(v, v')$ are adjacent if and only if either $u = v$ and $u'v' \in E(H)$, or $u' = v'$ and $uv \in E(G)$.

\subsection*{$\bullet$ A generalization of DP-coloring, signed coloring and list coloring}
By the definitions, DP-coloring is a special case of strictly $f$-degenerate transversal of $H$, which implies that strictly $f$-degenerate transversal is a generalization of DP-coloring. Dvo\v{r}\'{a}k and Postle \cite{MR3758240} have pointed out that DP-coloring is a generalization of list coloring; Kim and Ozeki \cite{MR3948125} showed that DP-coloring is also a generalization of signed coloring. Hence, strictly $f$-degenerate transversal is a generalization of DP-coloring, signed coloring and list coloring.

\subsection*{$\bullet$ A generalization of $(f_{1}, f_{2}, \dots, f_{s})$-partition and $L$-forested-coloring}
Next, we show that strictly $f$-degenerate transversal is a generalization of $(f_{1}, f_{2}, \dots, f_{s})$-partition. An \textbf{ID-cover} of a graph $G$ is a cover such that $(u, p)$ and $(v, q)$ are adjacent in $H$ if and only if $uv \in E(G)$ and $p = q$. In other words, an ID-cover of a graph $G$ is a Cartesian product of $G$ and an independent $s$-set. Note that ID-cover of $G$ is a special cover of $G$. We give a relation between the existence of a strictly $f$-degenerate transversal in an ID-cover $H$ and the $(f_{1}, f_{2}, \dots, f_{s})$-partitionability of $G$. 

\begin{proposition}
Let $(f_{1}, f_{2}, \dots, f_{s})$ be a sequence of nonnegative integer valued functions on $V(G)$, and $(H, f)$ be a valued ID-cover by $f(v, 1) = f_{1}(v)$, $f(v, 2) = f_{2}(v)$, $\dots$, $f(v, s) = f_{s}(v)$ for each $v \in V(G)$. Then $H$ has a strictly $f$-degenerate transversal if and only if $G$ is $(f_{1}, f_{2}, \dots, f_{s})$-partitionable.  
\end{proposition}
\begin{proof}
($\Longrightarrow$) Suppose that $R$ is a strictly $f$-degenerate transversal of $H$. For each $i \in [s]$, let $R_{i}$ be the set of vertices in $R$ such that they are colored with $i$, and let $V_{i} = \{v \mid (v, i) \in R\}$. Note that $(R_{1}, R_{2}, \dots, R_{s})$ is a partition of $R$ and $(V_{1}, V_{2}, \dots, V_{s})$ is a partition of $V(G)$. By the definition of strictly $f$-degenerate transversal, each induced subgraph $H[R_{i}]$ is strictly $f$-degenerate. On the other hand, $G[V_{i}] = H[R_{i}]$ is strictly $f_{i}$-degenerate for each $i \in [s]$. Hence, $G$ is $(f_{1}, f_{2}, \dots, f_{s})$-partitionable. 

($\Longleftarrow$) Suppose that $(V_{1}, V_{2}, \dots, V_{s})$ is a partition of $V(G)$ such that $G[V_{i}]$ is strictly $f_{i}$-degenerate for each $i \in [s]$. Let $R_{i} \coloneqq \{(v, i) \mid v \in V_{i}\}$ for each $i \in [s]$. Note that $R_{1} \cup R_{2} \cup \dots \cup R_{s}$ is a transversal of $H$ and $H[R_{i}] = G[V_{i}]$ for each $i \in [s]$. Note that for $i \neq j$, $H[R_{i} \cup R_{j}] = H[R_{i}] \cup H[R_{j}]$; \ie there is no edge between distinct subsets $R_{i}$ and $R_{j}$, thus $R_{1} \cup R_{2} \cup \dots \cup R_{s}$ is a strictly $f$-degenerate transversal. 
\end{proof}
Hence, strictly $f$-degenerate transversal is a generalization of $(f_{1}, f_{2}, \dots, f_{s})$-partition, and hence a generalization of $L$-forested-coloring. 

By the above arguments, we have the following result. 

\begin{proposition}
Let $G$ be a graph and $H$ be a cover of $G$. If $H$ has a strictly $f$-degenerate transversal whenever $f$ is a function from $V(H)$ to $\{0, 1, 2\}$, and $f(v, 1) + f(v, 2) + \dots + f(v, s) \geq k$ for each vertex $v \in V(G)$, then $G$ is DP-$k$-colorable, and the list vertex arboricity of $G$ is at most $\lceil\frac{k}{2}\rceil$. 
\end{proposition}

In \autoref{sec:DTR}, a degree type result is given, which generalizes Brooks' theorem, Gallai's theorem, degree-choosable, signed degree-colorable, and DP-degree-colorable results. In \autoref{sec:SELF-STRENGTHENING}, we show that the degree type result is self-strengthening. In \autoref{sec:MINIMAL}, some structural results are presented for a minimal counterexample for the existence of strictly $f$-degenerate transversal. In \autoref{sec:applications}, many applications of the results in \autoref{sec:DTR}--\ref{sec:MINIMAL} are presented. In the final section, we pose and discuss some open problems. 
\section{Degree type result}\label{sec:DTR}
The following result was obtained by Borodin \cite{MR0498204} and, independently, by Bollob\'{a}s and Manvel \cite{MR541961}. 
\begin{theorem}\label{BBM}
Let $G$ be a connected graph with maximum degree $\Delta(G) = \Delta \geq 3$ and not the complete graph $K_{\Delta + 1}$, and let $t_{1}, t_{2}, \dots, t_{s}$ be positive integers. If $s \geq 2$ and $t_{1} + t_{2} + \dots + t_{s} \geq \Delta$, then $G$ is $(t_{1}, t_{2}, \dots, t_{s})$-partitionable. 
\end{theorem}

Brooks' theorem follows from \autoref{BBM} by defining $t_{1} = t_{2} = \dots = t_{s} = 1$, and an upper bound on vertex arboricity follows from \autoref{BBM} by defining $t_{1} = t_{2} = \dots = t_{s} = 2$. An analogue to Brook's theorem in terms of list coloring was obtained by Vizing \cite{MR0498216} and, independently, by Erd\H{o}s, Rubin and Taylor \cite{MR593902}. 
\begin{theorem}\label{VERT}
Let $G$ be a connected graph with maximum degree $\Delta(G) = \Delta \geq 3$ and not the complete graph $K_{\Delta + 1}$. If $L$ is a list $\Delta$-assignment, then $G$ is $L$-colorable. 
\end{theorem}
Borodin, Kostochka and Toft \cite{MR1743629} gave a common generalization of \autoref{BBM} and \autoref{VERT} by replacing the constants with functions and forbidding some obstacles.  

\begin{theorem}\label{THM4}
Let $G$ be a connected graph with maximum degree $\Delta(G) = \Delta \geq 3$ and not the complete graph $K_{\Delta + 1}$. Let $f_{1}, f_{2}, \dots, f_{s}$ be functions from $V(G)$ to $\mathbb{Z}^{*}$. If $f_{1}(v) + f_{2}(v) + \dots + f_{s}(v) \geq \Delta$ for each $v \in V(G)$, and $G$ does not contain a monoblock, then $G$ is $(f_{1}, f_{2}, \dots, f_{s})$-partitionable. 
\end{theorem}

Borodin gave a characterization of degree-choosable graphs. 

\begin{theorem}\label{B}
Let $G$ be a connected graph with a list-assignment $L$ such that $|L(v)| \geq \deg(v)$ for each $v \in V(G)$. Then $G$ is not $L$-colorable if and only if $G$ is $R$-constructible. 
\end{theorem}

Borodin, Kostochka and Toft \cite{MR1743629} further extended \autoref{THM4} and \autoref{B} to the following result with a degree-condition 
\begin{equation*}
\text{$f_{1}(v) + f_{2}(v) + \dots + f_{s}(v) \geq \deg(v)$ for each $v \in V(G)$.}\tag{$\ast$}
\end{equation*}

\begin{theorem}\label{HARD}
Let $G$ be a connected graph and $f_{1}, f_{2}, \dots, f_{s}$ be functions from $V(G)$ to $\mathbb{Z}^{*}$ with degree-condition ($\ast$). Then $G$ is $(f_{1}, f_{2}, \dots, f_{s})$-partitionable if and only if $G$ is not hard-constructible.  
\end{theorem}

\begin{remark}
The exact definitions of ``monoblock'' in \autoref{THM4}, ``$R$-constructible'' in \autoref{B} and ``hard-constructible'' in \autoref{HARD} are not given here, so we refer the readers to \cite{MR1743629} for more details. 
\end{remark}

As proved in \cite{MR3484719} by M\'a\v{c}ajov\'a, Raspaud and \v{S}koviera, every signed graph $G$ satisfies $\chi^{\pm}(G) \leq \Delta(G) + 1$ and there are three classes of signed simple graphs for which the equality holds. Schweser and Stiebitz \cite{MR3612419} extended this Brooks' type result by characterizing the degree-choosable signed multigraphs. 

A graph $G$ is \textbf{DP-degree-colorable} if $(H, f)$ has a DP-coloring whenever $f$ is a function from $V(H)$ to $\{0, 1\}$ and $f(v, 1) + f(v, 2) + \dots + f(v, s) \geq \deg_{G}(v)$ for each $v \in V(G)$. A \textbf{GDP-tree} is a connected graph in which every block is either a cycle or a complete graph. Bernshteyn, Kostochka, and Pron \cite{MR3686937} gave a Brooks' type result for DP-coloring. A more detailed characterization of DP-degree-colorable multigraphs can be found in \cite{MR3948125}.  
\begin{theorem}[Bernshteyn, Kostochka, and Pron \cite{MR3686937}]\label{DL}
Let $G$ be a connected graph. Then $G$ is not DP-degree-colorable if and only if $G$ is a GDP-tree. 
\end{theorem}

In the following, we give a degree type result which can unify all the results mentioned in this section. Firstly, some preliminaries are presented. 
\begin{itemize}
\item The graph $\widetilde{K_{p}}$ is the Cartesian product of the complete graph $K_{p}$ and an independent $s$-set. 

\item The \textbf{circular ladder graph $\Gamma_{n}$} is the Cartesian product of the cycle $C_{n}$ and an independent set with two vertices. 

\item The \textbf{M\"{o}bius ladder} $M_{n}$ is the graph with vertex set $\big\{\,(i, j) \mid i \in [n], j \in [2]\,\big\}$, in which two vertices $(i, j)$ and $(i', j')$ are adjacent if and only if either
\begin{enumerate}[label = ---]
\item $i' = i + 1$ and $j = j'$ for $1 \leq i \leq n-1$, or 
\item $i = n$, $i'= 1$ and $j \neq j'$. 
\end{enumerate}
\end{itemize}

\begin{definition}
Let $(H, f)$ be a valued cover of a graph $G$. A \textbf{kernel} of $H$ is the subgraph obtained from $H$ by deleting each vertex $(u, j)$ with $f(u, j) = 0$. 
\end{definition}

\begin{definition}
A \textbf{building cover} is a valued cover $(H, f)$ of a $2$-connected graph $B$ satisfying
\begin{equation*}
f(v, 1) + f(v, 2) + \dots + f(v, s) = \deg_{B}(v)
\end{equation*} 
for each $v \in V(B)$ and at least one of the following holds:  
\begin{enumerate}[label = (\roman*)]
\item The kernel of $H$ is isomorphic to $B$. We call this cover a \textbf{monoblock}. 
\item If $B$ is isomorphic to a complete graph $K_{p}$ for some $p \geq 2$, then the kernel of $H$ is isomorphic to $\widetilde{K_{p}}$ with $f$ being constant on each component of $\widetilde{K_{p}}$. 
\item If $B$ is isomorphic to an odd cycle, then the kernel of $H$ is isomorphic to the circular ladder graph with $f = 1$. 
\item If $B$ is isomorphic to an even cycle, then the kernel of $H$ is isomorphic to the M\"{o}bius ladder with $f = 1$. \qed
\end{enumerate}
\end{definition}

\begin{definition}\label{constructible}
Every building cover is \textbf{constructible}. A valued cover $(H, f)$ of a graph $G$ is also \textbf{constructible} if it is obtained from a constructible valued cover $(H^{(1)}, f^{(1)})$ of $G^{(1)}$ and a constructible valued cover $(H^{(2)}, f^{(2)})$ of $G^{(2)}$ such that all of the following hold: 
\begin{enumerate}[label = (\roman*)]
\item the graph $G$ is obtained from $G^{(1)}$ and $G^{(2)}$ by identifying a vertex $w_{1}$ in $G^{(1)}$ and a vertex $w_{2}$ in $G^{(2)}$ as a new vertex $w$, and
\item the cover $H$ is obtained from $H^{(1)}$ and $H^{(2)}$ by identifying $(w_{1}, q)$ and $(w_{2}, q)$ as a new vertex $(w, q)$ for each $q \in [s]$, and  
\item $f(w, q) = f^{(1)}(w_{1}, q) + f^{(2)}(w_{2}, q)$ for each $q \in [s]$, $f = f^{(1)}$ on $H^{(1)} - X_{w}$, and $f = f^{(2)}$ on $H^{(2)} - X_{w}$. We simply write $\bm{f = f^{(1)} + f^{(2)}}$. \qed
\end{enumerate}
\end{definition}

By the definition of constructible cover, it is easy to obtain the following lemma. 
\begin{lemma}\label{FC}
If $(H, f)$ is a constructible cover, then $f(v, 1) + f(v, 2) + \dots + f(v, s) = \deg_{G}(v)$ for each $v \in V(G)$. 
\end{lemma}

\begin{lemma}\label{GE}
Let $G$ be a connected graph and $H$ be a valued cover with $f(v, 1) + f(v, 2) + \dots + f(v, s) \geq \deg_{G}(v)$ for each vertex $v \in V(G)$. If there exists a vertex $w$ such that $f(w, 1) + f(w, 2) + \dots + f(w, s) > \deg_{G}(w)$, then $H$ has a strictly $f$-degenerate transversal.  
\end{lemma}
\begin{proof}
Suppose that $G$ is a connected graph with minimum number of vertices such that there exist a valued cover $(H, f)$ satisfying the degree condition and having no strictly $f$-degenerate transversals. Note that each component of $G - w$ and related cover satisfies the condition of \autoref{GE}. By the minimality, $H - X_{w}$ has a strictly $f$-degenerate transversal $R$. Since $f(w, 1) + f(w, 2) + \dots + f(w, s) > \deg_{G}(w)$, there exists a vertex $(w, q)$ in $X_{w}$ such that it has less than $f(w, q)$ neighbors in $R$, which implies that $(w, q)$ together with $R$ is a strictly $f$-degenerate transversal of $H$, a contradiction. 
\end{proof}

Now, we are ready to formulate our main result. 
\begin{theorem}\label{MR}
Let $G$ be a connected graph and $(H, f)$ be a valued cover with $f(v, 1) + f(v, 2) + \dots + f(v, s) \geq \deg_{G}(v)$ for each vertex $v \in V(G)$. Then $H$ has a strictly $f$-degenerate transversal if and only if $(H, f)$ is non-constructible. 
\end{theorem}

To prove \autoref{MR}, we will use the method of Proof by Contradiction. The ``only if'' part is easy. For the ``if'' part, we first prove a minimal counterexample $G$ is $2$-connected, and then prove $G$ is either a complete graph or a cycle, and finally prove the cover is a building cover in each case to lead a contradiction. 

\begin{proof}[Proof of \autoref{MR}]
(``if'' part $\Longleftarrow$) Suppose that $G$ is a connected graph with minimum number of vertices such that there exists a non-constructible cover $(H, f)$ satisfying the degree condition and having no strictly $f$-degenerate transversals. By \autoref{GE}, we may assume that $f(v, 1) + f(v, 2) + \dots + f(v, s) = \deg_{G}(v)$ for each vertex $v \in V(G)$. 

\begin{lemma}\label{2-con}
The graph $G$ is $2$-connected. 
\end{lemma}
\begin{proof}
Suppose that $G$ is not $2$-connected, $G^{(1)} \cup G^{(2)} = G$ and $G^{(1)} \cap G^{(2)} = \{w\}$. We may assume that $H^{(1)} \cup H^{(2)} = H$ and $H^{(1)} \cap H^{(2)} = X_{w}$, where $H^{(1)}$ is the restriction of $H$ on $G^{(1)}$ and $H^{(2)}$ is the restriction of $H$ on $G^{(2)}$. By \autoref{GE}, $H - X_{w}$ has a strictly $f$-degenerate transversal $R$. Let $\alpha_{i}(w, q)$ denote the number of neighbors of $(w, q)$ in $R \cap H^{(i)}$ for each $q \in [s]$ and $i \in [2]$. We define two nonnegative integer valued functions $f^{(1)}$ on $H^{(1)}$ and $f^{(2)}$ on $H^{(2)}$ such that $f^{(i)}$ coincides with $f$ on $H^{(i)} - X_{w}$, $f^{(i)}(w, q) \geq \alpha_{i}(w, q)$ and $\sum_{q} f^{(i)}(w, q) = \deg_{G^{(i)}}(w)$ for each $i \in [2]$. Since $G$ is not constructible, at least one of $(H^{(1)}, f^{(1)})$ and $(H^{(2)}, f^{(2)})$, say $(H^{(1)}, f^{(1)})$, is not constructible. By the minimality of $G$, there exists a strictly $f^{(1)}$-degenerate transversal $R_{1}$ for $(H^{(1)}, f^{(1)})$. Note that an $f^{(1)}$-removing order of $R_{1}$ followed by an $f$-removing order of $R \cap V(H^{(2)})$ is an $f$-removing order of $R_{1} \cup (R \cap V(H^{(2)}))$. Then $R_{1}$ together with $R \cap V(H^{(2)})$ is a strictly $f$-degenerate transversal for $(H, f)$, a contradiction. 
\end{proof}
\begin{lemma}\label{NZ-PM}
For each edge $uv \in E(G)$, $\mathscr{M}_{uv}$ is a perfect matching between $X_{u}$ and $X_{v}$. Furthermore, we may assume that $f(w, t) > 0$ for each $w \in V(G)$ and $t \in [s]$. 
\end{lemma}
\begin{proof}
Let $u$ and $v$ be two adjacent vertices in $G$ and $f(u, p) > 0$ for some $p \in [s]$. Suppose that $(u, p)$ is not incident with any edge in $\mathscr{M}_{uv}$ or is adjacent to a vertex $(v, q)$ with $f(v, q) = 0$. We define a function $f'$ on $H - X_{u}$ by defining $f'(w, t) = f(w, t)$ for each $w \neq u$ and $t \in [s]$ except that, for each vertex $(w, t)$ adjacent to $(u, p)$, set $f'(w, t) = \max\{0, f(w, t) - 1\}$. Note that $G - u$ is connected and $f'(v, t) = f(v, t)$ for each $t \in [s]$, which implies that $f'(v, 1) + f'(v, 2) + \dots + f'(v, s) = \deg_{G}(v) > \deg_{G-u}(v)$. By \autoref{GE}, $H - X_{u}$ has a strictly $f'$-degenerate transversal $R$. Note that an $f'$-removing order of $R$ followed by $(u, p)$ is an $f$-removing order. Hence, $R$ together with $(u, p)$ is a strictly $f$-degenerate transversal for $(H, f)$, a contradiction. Therefore, for each edge $uv \in E(G)$, $\mathscr{M}_{uv}$ is a perfect matching between $X_{u}$ and $X_{v}$. 

By the above arguments, each component of $H$ has the same sign (zero or positive) for the value of $f$. Since $(w, t)$ can never be chosen in $R$ if $f(w, t) = 0$, we may assume that $f$ has positive value on each vertex of $H$.
\end{proof}
Since $G$ is $2$-connected and $H$ is not a monoblock, we may assume that $s \geq 2$ by \autoref{NZ-PM}. 
\begin{lemma}\label{COM}
Suppose that $w_{1}, w_{2}$ and $w$ are three distinct vertices of $G$ such that neither $G - \{w, w_{1}\}$ nor $G - \{w, w_{2}\}$ is disconnected. If $w_{1}$ and $w_{2}$ are in the same block $B$ of $G - w$, then $w$ is adjacent to either both or none of $w_{1}$ and $w_{2}$. 
\end{lemma}
\begin{proof}
Suppose that $ww_{1} \notin E(G)$ and $ww_{2} \in E(G)$. By \autoref{NZ-PM}, $\mathscr{M}_{ww_{2}}$ is a perfect matching between $X_{w}$ and $X_{w_{2}}$. Let $(w_{1}, 1)$ and $(w_{2}, q)$ be in the same component of $H[\bigcup_{v \in B}X_{v}]$. It is easy to see that $f(w_{2}, q) - 1 \neq f(w_{1}, 1)$ or $f(w_{2}, q) \neq f(w_{1}, 1)$. In the former case, we choose a vertex $(w, p)$ in $X_{w}$ adjacent to $(w_{2}, q)$ and reduce the value of $f$ for each neighbor of $(w, p)$ by one to obtain a valued cover $(H - X_{w}, f')$; in the later case, we choose a vertex $(w, p)$ in $X_{w}$ nonadjacent to $(w_{2}, q)$ and reduce the value of $f$ for each neighbor of $(w, p)$ by one to obtain a valued cover $(H - X_{w}, f')$. By construction, we have that $f'(w_{1}, 1) \neq f'(w_{2}, q)$. Since $s \geq 2$ and neither $w_{1}$ nor $w_{2}$ is a cut-vertex of $G - w$, the valued cover $(H - X_{w}, f')$ is non-constructible, thus there exists a strictly $f'$-degenerate transversal $R$ for $H - X_{w}$. Note that an $f'$-removing order of $R$ followed by $(w, p)$ is an $f$-removing order. Therefore, $R$ together with $(w, p)$ is a strictly $f$-degenerate transversal of $H$, a contradiction. 
\end{proof}
Take a vertex $w$ in $G$ with minimum degree $\delta = \delta(G)$; furthermore, if $G$ is not $\delta$-regular, then we choose $w$ adjacent to a vertex of degree greater than $\delta$. 

\begin{case}
The graph $G - w$ is $2$-connected.
\end{case}
By \autoref{COM}, $w$ is adjacent to all the vertices of $G - w$ and thus $G = K_{\delta + 1}$. 

\begin{subcase}
The function $f$ is not semi-constant. 
\end{subcase}
Suppose that there exists a component $Z$ of $H$ such that $f$ is not constant on it. Let $m$ be the minimum value of $f$ over all the vertices in $Z$. Hence, there exist two adjacent vertices $(u, p)$ and $(v, q)$ such that $f(u, p) = m$ and $f(v, q) > m$. Note that $f(v, 1) + f(v, 2) + \dots + f(v, s) = \delta$, thus $\delta \geq m + 1$. Let $S$ be a set of vertices consisting of $(v, q)$ and $m$ other neighbors of $(u, p)$ in the component $Z$. We obtain a function $f'$ on $H - X_{S}$ from $f$ by the following: for each vertex $(w, t)$ in $H - X_{S}$, $f'(w, t) = \max\big\{0, f(w, t) - \tau(w, t)\big\}$, where $X_{S}$ is the union of $X_{v}$ taken over all $v \in S$, and $\tau(w, t)$ is the number of neighbors of $(w, t)$ in $S$. Note that the degree of $u$ decreases by $m+1$, but 
\begin{equation*}
\sum_{t=1}^{s} f'(u, t) = \sum_{t=1}^{s} f(u, t) - m. 
\end{equation*}
Note that $G - S = K_{\delta - m}$ is connected. By \autoref{GE}, $H - X_{S}$ has a strictly $f'$-degenerate transversal $R$. Note that an $f'$-removing order of $R$ followed by $(v, q)$ and any order of other vertices in $S$ is an $f$-removing order of $R \cup S$  (here we need the condition $f(v, q) > m$). Hence, $R$ together with $S$ is a strictly $f$-degenerate transversal of $H$, a contradiction. 
\begin{subcase}
The function $f$ is semi-constant. 
\end{subcase}
Let $(w, q)$ be in a component $Z$ of $H$ and $f$ be constant on $Z$. Let $f'$ be obtained from $f$ by decreasing the output of $f$ by $1$ for each neighbor of $(w, q)$. By the previous arguments, $f'$ must be semi-constant on $H - X_{w}$, for otherwise $H - X_{w}$ has a strictly $f'$-degenerate transversal $R$, which implies that $R$ together with $(w, q)$ is a strictly $f$-degenerate transversal of $H$, a contradiction. Therefore, the closed neighborhood of $(w, q)$ induces a component of $H$ and each component of $H$ is isomorphic to $K_{\delta+1}$, which implies that $H$ is isomorphic to $\widetilde{K_{\delta + 1}}$, a contradiction. 
\begin{case}
The graph $G - w$ is not $2$-connected. 
\end{case}

As a consequence, $G - w$ has at least two end-blocks. Note that $\delta(G) = \delta$ and $\delta(G - w) \geq \delta - 1$, thus each end-block of $G - w$ has at least $\delta$ vertices. By \autoref{COM}, $w$ is adjacent to each non-cut vertex in the end-blocks of $G - w$, and $\deg_{G}(w) = \delta \geq 2(\delta -1)$, which implies that $\delta \leq 2$. Recall that $\deg_{G}(w) = f(w, 1) + f(w, 2) + \dots + f(w, s) \geq s \geq 2$, thus $\delta = s = 2$. Moreover, $G - w$ has only two end-blocks, each one is isomorphic to $K_{2}$. By the choice of $w$, the graph $G$ is $2$-regular, hence it is a cycle and $f = 1$. By \autoref{NZ-PM}, $H$ is a circular ladder graph or M\"{o}bius ladder. It is not hard to see that, if $G$ is an odd cycle and $H$ is a M\"{o}bius ladder, or $G$ is an even cycle and $H$ is a circular ladder graph, then $H$ has a strictly $1$-degenerate transversal (independent set) of $H$, a contradiction. 

\medskip
(``only if'' part $\Longrightarrow$) Suppose that $G$ is a connected graph with minimum number of vertices such that $(H, f)$ is constructible and $H$ has a strictly $f$-degenerate transversal. If $H$ is a monoblock, then the kernel of $H$ is isomorphic to $G$, but $G$ is not strictly $f$-degenerate, a contradiction. Suppose that $G$ is isomorphic to $K_{p}$, $H$ is isomorphic to $\widetilde{K_{p}}$, $f$ is semi-constant on $H$ and $H$ has a strictly $f$-degenerate transversal $R$. Since $f(v, 1) + f(v, 2) + \dots + f(v, s) = p - 1$ for each $v \in V(G)$ and $|V(G)| = p$, there exists a color class $R_{i}$ such that $|R_{i}| > f(v, i)$ and $H[R_{i}]$ is a complete graph, thus it is not strictly $f$-degenerate, a contradiction. If $G$ is isomorphic to an odd cycle, then the kernel of $H$ is isomorphic to the circular ladder graph with $f = 1$; if $G$ is isomorphic to an even cycle, then the kernel of $H$ is isomorphic to the M\"{o}bius ladder with $f = 1$. It is easy to check that $H$ has no strictly $1$-degenerate transversal in these two situations, a contradiction. 

Now, suppose that a constructible cover $(H, f)$ of $G$ is obtained from a constructible cover $(H^{(1)}, f^{(1)})$ of $G^{(1)}$ and a constructible cover $(H^{(2)}, f^{(2)})$ of $G^{(2)}$ as in \autoref{constructible}. Let $R$ be a strictly $f$-degenerate transversal of $H$ and suppose without loss of generality $(w, 1) \in R$. By the minimality of $G$, $R \cap H^{(1)}$ is not a strictly $f^{(1)}$-degenerate transversal, but $R \cap (H^{(1)} - X_{w})$ is a strictly $f^{(1)}$-degenerate transversal of $H^{(1)} - X_{w}$, thus there exists a subset $R_{1}$ of $R \cap H^{(1)}$ such that $(w, 1) \in R_{1}$ and $\deg_{R_{1}}(x) \geq f^{(1)}(x)$ for each $x \in R_{1}$. Similarly, there exists a subset $R_{2}$ of $R \cap H^{(2)}$ such that $(w, 1) \in R_{2}$ and $\deg_{R_{2}}(x) \geq f^{(2)}(x)$ for each $x \in R_{2}$. Let $R' \coloneqq R_{1} \cup R_{2}$. Note that $\deg_{R'}(x) \geq f(x)$ for each $x \in R'$, this contradicts the fact that $R$ is a strictly $f$-degenerate transversal. 
\end{proof}

\section{Strengthened results}\label{sec:SELF-STRENGTHENING}
Borodin, Kostochka and Toft \cite{MR1743629} showed that \autoref{HARD} is self-strengthening, and gave the following stronger form.
\begin{theorem}[Borodin, Kostochka and Toft \cite{MR1743629}]\label{SELF}
Let $G$ be a connected graph and $f_{1}, f_{2}, \dots, f_{s}$ be functions from $V(G)$ to $\mathbb{Z}^{*}$ with $f_{1}(v) + f_{2}(v) + \dots + f_{s}(v) \geq \deg(v)$ for each $v \in V(G)$. Then $G$ can be partitioned into strictly $f_{i}$-degenerate induced subgraphs $G_{i}$ with $\deg_{G_{i}}(v) \leq f_{i}(v)$ whenever $1 \leq i \leq s$ and $v \in V(G_{i})$ if and only if $G$ is not hard-constructible.  
\end{theorem}

Similar to \autoref{SELF}, \autoref{MR} is also self-strengthening. Let $R$ be a strictly $f$-degenerate transversal of a valued cover $(H, f)$, we define $\mathrm{def}(R)$ as the following, 
\begin{equation*}
\mathrm{def}(R) \coloneqq |E(H[R])| - \sum_{(v, q) \in R} f(v, q). 
\end{equation*}
\begin{theorem}\label{SMR}
Let $G$ be a connected graph and $(H, f)$ be a valued cover with $f(v, 1) + f(v, 2) + \dots + f(v, s) \geq \deg_{G}(v)$ for each $v \in V(G)$. Then $H$ has a strictly $f$-degenerate transversal $R$ such that $\deg_{R}(v, q) \leq f(v, q)$ for each $v \in V(G)$ and $q \in [s]$, if and only if $H$ is non-constructible. 
\end{theorem}
\begin{proof}
By \autoref{MR}, it suffices to prove that, if $H$ has a strictly $f$-degenerate transversal, then it has a strictly $f$-degenerate transversal $R$ such that $\deg_{R}(v, q) \leq f(v, q)$ for each $v \in V(G)$ and $q \in [s]$. Let 
\begin{equation*}
\mathrm{def}(R^{*}) = \min_{R \in \mathcal{R}} \big\{\mathrm{def}(R) \big\}, 
\end{equation*}
where $\mathcal{R}$ is the set of all the strictly $f$-degenerate transversals of $H$. Next, we show that $R^{*}$ is the desired set. Suppose that there exists a vertex $(w, p) \in R^{*}$ with $\deg_{R^{*}}(w, p) > f(w, p)$. Let $D$ be obtained from $R^{*}$ by deleting $(w, p)$. Note that 
\begin{equation*}
\deg_{D}(w, 1) + \deg_{D}(w, 2) + \dots + \deg_{D}(w, s) \leq \deg_{G}(w) \leq f(w, 1) + f(w, 2) + \dots + f(w, s), 
\end{equation*}
thus there exists a vertex $(w, q)$ with $\deg_{D}(w, q) < f(w, q)$. Let $R'$ be obtained from $R^{*}$ by replacing $(w, p)$ with $(w, q)$. It is easy to see that $R' \in \mathcal{R}$. But $\mathrm{def}(R') = \mathrm{def}(R^{*}) - \deg_{D}(w, p) + f(w, p) + \deg_{D}(w, q)  - f(w, q) < \mathrm{def}(R^{*})$, which contradicts the minimality of $\mathrm{def}(R^{*})$. 
\end{proof}

When $f(v, 1) + f(v, 2) + \dots + f(v, s) > \deg_{G}(v)$ for each $v \in V(G)$, we have a more specified strictly $f$-degenerate transversal in a particularly strong sense. 
\begin{theorem}\label{MSMR}
Let $G$ be a connected graph and $(H, f)$ be a valued cover of $G$. If $f(v, 1) + f(v, 2) + \dots + f(v, s) > \deg_{G}(v)$ for each $v \in V(G)$, then $H$ has a strictly $f$-degenerate transversal $R$ such that $\deg_{R}(v, q) < f(v, q)$ for each $v \in V(G)$ and $q \in [s]$. 
\end{theorem}
\begin{proof}
By \autoref{GE}, $H$ has a strictly $f$-degenerate transversal. It suffices to prove that $H$ has a strictly $f$-degenerate transversal $R$ such that $\deg_{R}(v, q) < f(v, q)$ for each $v \in V(G)$ and $q \in [s]$. Let 
\begin{equation*}
\mathrm{def}(R^{*}) = \min_{R \in \mathcal{R}} \big\{\mathrm{def}(R) \big\}, 
\end{equation*}
where $\mathcal{R}$ is the set of all the strictly $f$-degenerate transversal of $H$. Next, we show that $R^{*}$ is the desired set. Suppose that there exists a vertex $(w, p) \in R^{*}$ with $\deg_{R^{*}}(w, p) \geq f(w, p)$. Let $D$ be obtained from $R^{*}$ by deleting $(w, p)$. Note that 
\begin{equation*}
\deg_{D}(w, 1) + \deg_{D}(w, 2) + \dots + \deg_{D}(w, s) \leq \deg_{G}(w) < f(w, 1) + f(w, 2) + \dots + f(w, s), 
\end{equation*}
thus there exists a vertex $(w, q)$ with $\deg_{D}(w, q) < f(w, q)$. Let $R'$ be obtained from $R^{*}$ by replacing $(w, p)$ with $(w, q)$. It is easy to see that $R' \in \mathcal{R}$. But $\mathrm{def}(R') = \mathrm{def}(R^{*}) - \deg_{D}(w, p) + f(w, p) + \deg_{D}(w, q)  - f(w, q) < \mathrm{def}(R^{*})$, which contradicts the minimality of $\mathrm{def}(R^{*})$. 
\end{proof}

\section{Gallai type result}\label{sec:MINIMAL}
Critical graphs in various coloring problems have been extensively studied, in this section we give some structural results on critical graphs with respect to strictly $f$-degenerate transversals. 

A graph $G$ is \textbf{$k$-critical} if $\chi(G) = k$ but $\chi(G - v) \leq k - 1$ for all $v \in V(G)$. A \textbf{Gallai tree} is a connected graph in which every block is either a complete graph or an odd cycle. The following structural result was originally proved by Gallai. 
\begin{theorem}[Gallai \cite{MR0188099}]\label{Critical-1}
Let $k \geq 3$ and let $G$ be a $(k + 1)$-critical graph. Let 
\begin{equation*}
D \coloneqq \{v \in V(G) \mid \deg_{G}(v) = k\}.
\end{equation*}
Then each component of $G[D]$ is a Gallai tree. 
\end{theorem}

A graph $G$ is \textbf{$L$-critical} if $G$ does not have an $L$-coloring, but every proper subgraph does. For $L$-critical graphs, Kostochka \etal gave the following more general result. 

\begin{theorem}[Kostochka \etal \cite{MR1425799}]\label{Critical-2}
Let $k \geq 3$ and let $G$ be a graph. Suppose that $L$ is a list $k$-assignment for $G$ such that $G$ is $L$-critical, and 
\begin{equation*}
D \coloneqq\{v \in V(G) \mid \deg_{G}(v) = k\}. 
\end{equation*} 
Then each component of $G[D]$ is a Gallai tree. 
\end{theorem}

A graph $G$ is \textbf{DP-$(k + 1)$-critical} if $\chi_{\mathrm{DP}}(G) = k + 1$ and $\chi_{\mathrm{DP}}(G - v) = k$ for all $v \in V(G)$. 

\begin{theorem}[Bernshteyn and Kostochka \cite{MR3686937}]\label{Critical-3}
Let $k \geq 3$ and let $G$ be a graph. Suppose that $G$ is a DP-$(k + 1)$-critical graph, and 
\begin{equation*}
D \coloneqq\{v \in V(G) \mid \deg_{G}(v) = k\}. 
\end{equation*} 
Then each component of $G[D]$ is a GDP-tree. 
\end{theorem}

Let $G$ be a graph and $(H, f)$ be a valued cover of $G$. The pair $(H, f)$ is \textbf{minimal non-strictly $f$-degenerate} if $H$ has no strictly $f$-degenerate transversal, but $(H - X_{v}, g)$ has a strictly $g$-degenerate transversal for each $v \in V(G)$, where $g$ is the restriction of $f$ on $V(H) - X_{v}$. We will give some structural results on minimal non-strictly $f$-degenerate pair $(H, f)$. 

The following result is about the minimum degree of the critical graphs. 
\begin{theorem}\label{L}
Let $G$ be a graph and $(H, f)$ be a valued cover of $G$. If $(H, f)$ is a minimal non-strictly $f$-degenerate pair, then $G$ is connected and $f(v, 1) + f(v, 2) + \dots + f(v, s) \leq \deg_{G}(v)$ for each $v \in V(G)$. 
\end{theorem}
\begin{proof}
Note that $G$ is connected. Suppose that $G$ has a vertex $w$ such that $f(w, 1) + f(w, 2) + \dots + f(w, s) > \deg_{G}(w)$. By the minimality, $H - X_{w}$ has a strictly $f$-degenerate transversal $T$. Note that 
\begin{equation*}
\deg_{T}(w, 1) + \deg_{T}(w, 2) + \dots + \deg_{T}(w, s) \leq \deg_{G}(w) < f(w, 1) + f(w, 2) + \dots + f(w, s),
\end{equation*}
thus there exists a vertex $(w, q)$ such that $\deg_{T}(w, q) < f(w, q)$. Hence, $(w, q)$ together with $T$ is a strictly $f$-degenerate transversal of $H$, a contradiction. 
\end{proof}

The next theorem is an analogue to \autoref{Critical-1}, \ref{Critical-2} and \ref{Critical-3}. Let 
\begin{equation*}
\mathscr{D} \coloneqq \{\,v \mid f(v, 1) + f(v, 2) + \dots + f(v, s) \geq \deg_{G}(v)\,\}.
\end{equation*}

\begin{theorem}
Let $G$ be a graph and $(H, f)$ be a valued cover of $G$. If $(H, f)$ is a minimal non-strictly $f$-degenerate pair, $B$ is a nonempty subset of $\mathscr{D}$ with $G[B]$ having no cut vertex, then $G[B]$ is a cycle, or a complete graph, or $\deg_{G[B]}(v) \leq \max_{q} \big\{f(v, q)\big\}$ for each $v \in B$. 
\end{theorem}
\begin{proof}
By the minimality, $\big(H - \bigcup_{v \in B}X_{v}, f\big)$ has a strictly $f$-degenerate transversal $T$. Let $H' = H_{B}$ and $\tau(v, q)$ be the number of neighbors  of $(v, q)$ in $T$. We define a new function $f'$ on $H'$ by $f'(v, q) = \max\{0, f(v, q) - \tau(v, q)\}$ for each $v \in B$ and $q \in [s]$. Note that 
\begin{equation*}
f'(v, 1) + f'(v, 2) + \dots + f'(v, s) \geq f(v, 1) + f(v, 2) + \dots + f(v, s) - \deg_{G - B}(v) \geq \deg_{G}(v) - \deg_{G - B}(v) = \deg_{G[B]}(v) 
\end{equation*}
for each $v \in B$. If $H'$ has a strictly $f'$-degenerate transversal $T'$, then an $f'$-removing order of $T'$ followed by an $f$-removing order of $T$ is an $f$-removing order of $T' \cup T$, which implies that $T'\cup T$ is a strictly $f$-degenerate transversal of $H$, a contradiction. So we may assume that $H'$ has no strictly $f'$-degenerate transversal. By \autoref{MR}, $(H', f')$ is constructible. Since $G[B]$ is $2$-connected, this implies that $G[B]$ is a cycle or a complete graph or $(H', f')$ is a monoblock. Moreover, if $(H', f')$ is a monoblock, then for each $v \in B$, there exists a vertex $(v, q_{v})$ in $V(H')$ such that $\deg_{G[B]}(v) = f'(v, q_{v}) \leq \max_{q} \big\{f(v, q)\big\}$. 
\end{proof}

\section{Applications}\label{sec:applications}
In this section, we give several applications of our main results, some of them are new and others are known. Using our results in the previous sections, some proofs are very short, so we omit them. 

\begin{theorem}\label{THM-5.1}
Let $m$ be an integer. If $G$ is a strictly $m$-degenerate graph and $(H, f)$ is a valued cover of $G$ with $f(v, 1) + f(v, 2) + \dots + f(v, s) \geq m$ for each $v \in V(G)$, then $H$ has a strictly $f$-degenerate transversal. 
\end{theorem}
\begin{proof}
Suppose that $G$ is a counterexample with the number of vertices as small as possible. Note that $(H, f)$ is a minimal non-strictly $f$-degenerate pair. By \autoref{L}, the minimum degree of $G$ is at least $m$, but the strictly $m$-degenerate graph $G$ has minimum degree less than $m$, a contradiction. 
\end{proof}

This immediately implies the following results. 
\begin{theorem}[Dvo\v{r}\'{a}k and Postle \cite{MR3758240}]
Every strictly $m$-degenerate graph is DP-$m$-colorable. 
\end{theorem}
\begin{theorem}
If $G$ is a strictly $m$-degenerate graph and $t_{1} + t_{2} + \dots + t_{s} \geq m$, then $G$ is $(t_{1}, t_{2}, \dots, t_{s})$-partitionable. 
\end{theorem}
Duffin \cite{MR0175809} showed that every $K_{4}$-minor-free graph is strictly $3$-degenerate, so we have the following result. 
\begin{theorem}
Every $K_{4}$-minor-free graph can be partitioned into an independent set and an induced forest. 
\end{theorem}

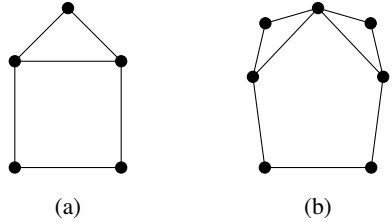
\begin{figure}[htbp]%
\centering
\subcaptionbox{\label{fig:subfig:a-}}
{\begin{tikzpicture}[scale = 0.8]
\coordinate (A) at (45:1);
\coordinate (B) at (135:1);
\coordinate (C) at (225:1);
\coordinate (D) at (-45:1);
\coordinate (H) at (90:1.414);
\draw (A)--(H)--(B)--(C)--(D)--cycle;
\draw (A)--(B);
\node[circle, inner sep = 1.5, fill, draw] () at (A) {};
\node[circle, inner sep = 1.5, fill, draw] () at (B) {};
\node[circle, inner sep = 1.5, fill, draw] () at (C) {};
\node[circle, inner sep = 1.5, fill, draw] () at (D) {};
\node[circle, inner sep = 1.5, fill, draw] () at (H) {};
\end{tikzpicture}}\hspace{1.5cm}
\subcaptionbox{\label{fig:subfig:b-}}
{\begin{tikzpicture}[scale = 0.8]
\coordinate (A) at (30:1);
\coordinate (B) at (150:1);
\coordinate (C) at (225:1);
\coordinate (D) at (-45:1);
\coordinate (H) at (90:1.414);
\coordinate (X) at (60:1.4);
\coordinate (Y) at (120:1.4);
\draw (A)--(X)--(H)--(Y)--(B)--(C)--(D)--cycle;
\draw (A)--(H)--(B);
\node[circle, inner sep = 1.5, fill, draw] () at (A) {};
\node[circle, inner sep = 1.5, fill, draw] () at (B) {};
\node[circle, inner sep = 1.5, fill, draw] () at (C) {};
\node[circle, inner sep = 1.5, fill, draw] () at (D) {};
\node[circle, inner sep = 1.5, fill, draw] () at (H) {};
\node[circle, inner sep = 1.5, fill, draw] () at (X) {};
\node[circle, inner sep = 1.5, fill, draw] () at (Y) {};
\end{tikzpicture}}
\caption{Forbidden configurations in \autoref{MRONEPLANAR} and \autoref{MRONE}.}
\label{NOA}
\end{figure}

Li and Wang \cite{Li2019} investigated some planar/toroidal graphs without certain small subgraphs. 

\begin{restatable}[Li and Wang \cite{Li2019}]{theorem}{MRONEPLANAR}\label{MRONEPLANAR}
Let $G$ be a planar graph without subgraphs isomorphic to the configurations in \autoref{NOA}, and let $(H, f)$ be a valued cover of $G$. If $f(v, 1) + f(v, 2) + \dots + f(v, s) \geq 4$ for each $v \in V(G)$, then $H$ has a strictly $f$-degenerate transversal. 
\end{restatable}

\begin{restatable}[Li and Wang \cite{Li2019}]{theorem}{MRONE}\label{MRONE}
Let $G$ be a toroidal graph without subgraphs isomorphic to the configurations in \autoref{NOA}, and let $(H, f)$ be a valued cover of $G$. If $f(v, 1) + f(v, 2) + \dots + f(v, s) \geq 4$ for each $v \in V(G)$, and $(H_{S}, f)$ is not a monoblock whenever $G[S]$ is a 2-connected $4$-regular graph, then $H$ has a strictly $f$-degenerate transversal. 
\end{restatable}

\begin{corollary}
If $G$ is a toroidal graph without subgraphs isomorphic to the configurations in \autoref{NOA}, then it is $(1, 3)$-partitionable and $(2, 2)$-partitionable. 
\end{corollary}

Two cycles are \textbf{adjacent} if they have at least one common edge, and two cycles are \textbf{intersecting} if they have at least one vertex in common. Every configuration in \autoref{NOA} has a $5$-cycle adjacent to a $3$-cycle, so we immediately have the following two known results. 
\begin{theorem}[Huang, Chen and Wang \cite{MR3320048}]
Every toroidal graph without $5$-cycles adjacent to $3$-cycles has list vertex arboricity at most two.
\end{theorem}
\begin{theorem}[Zhang \cite{MR3570576}]
Every toroidal graph without $5$-cycles has list vertex arboricity at most two.
\end{theorem}

\begin{figure}%
\centering
\subcaptionbox{\label{fig:subfig:a--}}
{\begin{tikzpicture}[scale = 0.8]
\coordinate (A) at (45:1);
\coordinate (B) at (135:1);
\coordinate (C) at (225:1);
\coordinate (D) at (-45:1);
\coordinate (H) at (90:1.414);
\draw (A)--(H)--(B)--(C)--(D)--cycle;
\draw (A)--(B);
\node[circle, inner sep = 1.5, fill, draw] () at (A) {};
\node[circle, inner sep = 1.5, fill, draw] () at (B) {};
\node[circle, inner sep = 1.5, fill, draw] () at (C) {};
\node[circle, inner sep = 1.5, fill, draw] () at (D) {};
\node[circle, inner sep = 1.5, fill, draw] () at (H) {};
\end{tikzpicture}}\hspace{1.5cm}
\subcaptionbox{\label{fig:subfig:b--}}
{\begin{tikzpicture}[scale = 0.8]
\coordinate (A) at (30:1);
\coordinate (B) at (150:1);
\coordinate (C) at (225:1);
\coordinate (D) at (-45:1);
\coordinate (H) at (90:1.414);
\coordinate (X) at (60:1.4);
\coordinate (Y) at (120:1.4);
\coordinate (T) at ($(H)!(A)!(X)$);
\coordinate (Z) at ($(A)!2!(T)$);
\draw (A)--(X)--(Z)--(H)--(Y)--(B)--(C)--(D)--cycle;
\draw (A)--(H)--(B);
\draw (H)--(X);
\node[circle, inner sep = 1.5, fill, draw] () at (A) {};
\node[circle, inner sep = 1.5, fill, draw] () at (B) {};
\node[circle, inner sep = 1.5, fill, draw] () at (C) {};
\node[circle, inner sep = 1.5, fill, draw] () at (D) {};
\node[circle, inner sep = 1.5, fill, draw] () at (H) {};
\node[circle, inner sep = 1.5, fill, draw] () at (X) {};
\node[circle, inner sep = 1.5, fill, draw] () at (Y) {};
\node[circle, inner sep = 1.5, fill, draw] () at (Z) {};
\end{tikzpicture}}\hspace{1.5cm}
\subcaptionbox{\label{fig:subfig:c--}}
{\begin{tikzpicture}[scale = 0.8]
\coordinate (A) at (30:1);
\coordinate (B) at (150:1);
\coordinate (C) at (225:1);
\coordinate (D) at (-45:1);
\coordinate (H) at (90:1.414);
\coordinate (X) at (60:1.4);
\coordinate (Y) at (120:1.4);
\coordinate (T) at ($(A)!(H)!(X)$);
\coordinate (Z) at ($(H)!2!(T)$);
\draw (A)--(Z)--(X)--(H)--(Y)--(B)--(C)--(D)--cycle;
\draw (X)--(A)--(H)--(B);
\node[circle, inner sep = 1.5, fill, draw] () at (A) {};
\node[circle, inner sep = 1.5, fill, draw] () at (B) {};
\node[circle, inner sep = 1.5, fill, draw] () at (C) {};
\node[circle, inner sep = 1.5, fill, draw] () at (D) {};
\node[circle, inner sep = 1.5, fill, draw] () at (H) {};
\node[circle, inner sep = 1.5, fill, draw] () at (X) {};
\node[circle, inner sep = 1.5, fill, draw] () at (Y) {};
\node[circle, inner sep = 1.5, fill, draw] () at (Z) {};
\end{tikzpicture}}
\caption{Forbidden configurations in \autoref{MRTHREE}.}
\label{E}
\end{figure}
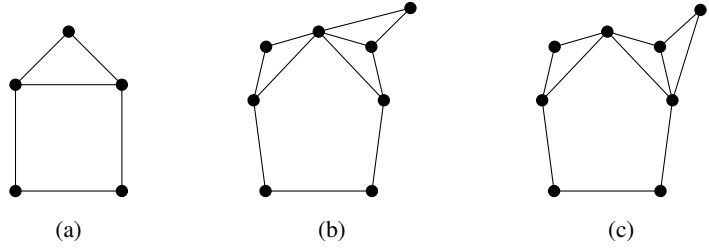

Note that there are mutually adjacent $3$-, $4$- and $5$-cycles in each configuration of \autoref{E}. Furthermore, \autoref{fig:subfig:a-} is the same with \autoref{fig:subfig:a--}, while \autoref{fig:subfig:b-} is a subgraph of \autoref{fig:subfig:b--} and \autoref{fig:subfig:c--}, so the class of planar graphs without the configurations in \autoref{E} is a larger class, thus they are forbidden in the next theorem. 
\begin{restatable}[Li and Wang \cite{Li2019}]{theorem}{MRTHREE}\label{MRTHREE}
Let $G$ be a planar graph without subgraphs isomorphic to the configurations in \autoref{E}. Let $H$ be a cover of $G$ and $f$ be a function from $V(H)$ to $\{0, 1, 2\}$. If $f(v, 1) + f(v, 2) + \dots + f(v, s) \geq 4$ for each $v \in V(G)$, then $H$ has a strictly $f$-degenerate transversal. 
\end{restatable}

The following \autoref{SN1}--\ref{NO4} are direct consequences of \autoref{MRTHREE}. 
\begin{theorem}[Sittitrai and Nakprasit \cite{MR4345150}]\label{SN1}
Let $G$ be a planar graph without $4$-cycles adjacent to $3$-cycles. Let $H$ be a cover of $G$ and $f$ be a function from $V(H)$ to $\{0, 1, 2\}$. If $f(v, 1) + f(v, 2) + \dots + f(v, s) \geq 4$ for each $v \in V(G)$, then $H$ has a strictly $f$-degenerate transversal. 
\end{theorem}

\begin{theorem}[Sittitrai and Nakprasit \cite{MR4345150}]\label{SN2}
Let $G$ be a planar graph without $4$-cycles adjacent to $3$-cycles. If $H$ is a cover of $G$ and $s \geq 2$, then $H$ has a strictly $2$-degenerate transversal. 
\end{theorem}

When we give an ID-cover, \autoref{SN1} implies the following results by Borodin and Ivanova \cite{MR2566928}. 
\begin{theorem}[Borodin and Ivanova \cite{MR2566928}]
Let $G$ be a planar graph without $4$-cycles adjacent to $3$-cycles. If $f_{1}(v) + f_{2}(v) + \dots + f_{s}(v)  \geq 4$ for each $v \in V(G)$, and $f_{i}(v) \in \{0, 1, 2\}$ for each $v \in V(G)$ and $i \in [s]$, then $G$ is $(f_{1}, f_{2}, \dots, f_{s})$-partitionable. 
\end{theorem}

\begin{theorem}
\text{}
\begin{enumerate}[label = (\roman*)]
\item (Borodin and Ivanova \cite{MR2586624}) Every planar graph without $4$-cycles adjacent to $3$-cycles is $4$-choosable. 
\item (Borodin and Ivanova \cite{MR2566928}) Every planar graph without $4$-cycles adjacent to $3$-cycles has list vertex arboricity at most two. \qed
\end{enumerate}
\end{theorem}

\begin{theorem}[Raspaud and Wang \cite{MR2408378}]\label{NO4}
Every planar graph without $4$-cycles has vertex arboricity at most two. 
\end{theorem}

For an embedded toroidal graph $G$, 
\begin{equation}\label{TO}
\sum_{v \in V(G)} (\deg(v) - 6) + \sum_{f \in F(G)} (2\deg(f) - 6) = 0. \tag{$\ast\ast$}
\end{equation}
Then every toroidal graph is strictly $7$-degenerate. 
\begin{theorem}\label{TT}
Let $G$ be a toroidal graph and $(H, f)$ be a valued cover. 
\begin{enumerate}[label = (\roman*)]
\item If $f(v, 1) + f(v, 2) + \dots + f(v, s) \geq 7$ for each $v \in V(G)$, then $H$ has a strictly $f$-degenerate transversal. 
\item If $f(v, 1) + f(v, 2) + \dots + f(v, s) \geq 6$ for each $v \in V(G)$, $K_{7}$ is not a subgraph of $G$ and $(H_{S}, f)$ is not a monoblock whenever $G[S]$ is a 2-connected $6$-regular graph, then $H$ has a strictly $f$-degenerate transversal. 
\end{enumerate}
\end{theorem}
\begin{proof}
(i) This follows from \autoref{THM-5.1} and the fact that $G$ is strictly $7$-degenerate. 

(ii) Suppose that $G$ is a counterexample to \autoref{TT} with minimum number of vertices. Note that $G$ is connected and $(H, f)$ is a minimal non-strictly $f$-degenerate pair. By \autoref{L}, the minimum degree of $G$ is at least $6$, this together with \eqref{TO} implies that $G$ is $6$-regular and every face is a $3$-face. Hence, $G$ is a $2$-connected $6$-regular graph, which contradicts \autoref{MR} and the constructions of the building covers. 
\end{proof}
The following result is a direct consequence of \autoref{TT} and it answers a question posed in \cite{MR3508765}. 
\begin{theorem}[Wang, Chen and Wang \cite{MR3862632}]
If $G$ is a toroidal graph, then the list vertex arboricity is at most four. Moreover, the list vertex arboricity is four if and only if $K_{7}$ is a subgraph of $G$. 
\end{theorem}

Chartrand and Kronk \cite{MR0239996} showed that the vertex arboricity of a strictly $m$-degenerate graph is at most $\lceil\frac{m}{2}\rceil$. Xue and Wu \cite{MR2947381} extended this result to the list context. 
\begin{theorem}[Xue and Wu \cite{MR2947381}]\label{K-D}
The list vertex arboricity of a strictly $m$-degenerate graph is at most $\lceil\frac{m}{2}\rceil$. 
\end{theorem}
\begin{proof}
Let $f_{1}, f_{2}, \dots, f_{s}$ be functions from $V(G)$ to $\{0, 2\}$, and let $(H, f)$ be a valued ID-cover of $G$ with $f(v, 1) = f_{1}(v)$, $f(v, 2) = f_{2}(v), \dots, f(v, s) = f_{s}(v)$ for each $v \in V(G)$. Suppose that $f_{1}(v) + f_{2}(v) + \dots + f_{s}(v) \geq 2\lceil\frac{m}{2}\rceil$. This implies that $f(v, 1) + f(v, 2) + \dots + f(v, s) \geq 2\lceil\frac{m}{2}\rceil \geq m$ for each $v \in V(G)$. By \autoref{THM-5.1}, $H$ has a strictly $f$-degenerate transversal, which implies that the list vertex arboricity of $G$ is at most $\lceil\frac{m}{2}\rceil$. 
\end{proof}

By Euler's formula, every planar graph is strictly $6$-degenerate, thus we immediately obtain the following result. 
\begin{theorem}[Chartrand, Kronk and Wall \cite{MR0236049}]
Every planar graph has vertex arboricity at most three. 
\end{theorem}

Later, Chartrand and Kronk \cite{MR0239996} constructed an example of a planar graph with vertex arboricity three to show the bound is sharp. Note that every graph with maximum degree $\Delta$ is $\Delta$-degenerate, thus \autoref{K-D} implies the following result.

\begin{theorem}[Chartrand, Kronk and Wall \cite{MR0236049}, Chartrand and Kronk \cite{MR0239996}]
Every graph with maximum degree $\Delta$ has vertex arboricity at most $\lceil\frac{\Delta + 1}{2}\rceil$.
\end{theorem}

Matsumoto \cite{MR1037427} showed that the linear vertex arboricity of $G$ is also at most $\lceil\frac{\Delta + 1}{2}\rceil$. These two results can be strengthened to the linear list vertex arboricity. 
\begin{theorem}
Every graph with maximum degree $\Delta$ has linear list vertex arboricity at most $\lceil\frac{\Delta + 1}{2}\rceil$.
\end{theorem}
Actually, it is an immediate consequence of the following result, which is a special case of \autoref{MSMR}. 
\begin{theorem}
Let $G$ be a connected graph and $(H, f)$ be a valued cover of $G$. If $f$ is a function from $V(H)$ to $\{0, 1, 2\}$ and $f(v, 1) + f(v, 2) + \dots + f(v, s) > \deg_{G}(v)$ for each $v \in V(G)$, then $H$ has a strictly $f$-degenerate transversal $R$ such that the maximum degree of $H[R]$ is at most one. 
\end{theorem}

An analogue of Brooks' Theorem for vertex arboricity was obtained by Kronk and Mitchem \cite{MR0360334}. 
\begin{theorem}[Kronk and Mitchem \cite{MR0360334}]
Let $G$ be a connected graph with maximum degree $\Delta$. If it is neither a complete graph of odd order nor a cycle, then the vertex arboricity is at most $\lceil\frac{\Delta}{2} \rceil$. 
\end{theorem}

Borodin, Kostochka and Toft \cite{MR1743629} obtained the following Brooks' type result for linear list vertex arboricity. 
\begin{theorem}[Borodin, Kostochka and Toft \cite{MR1743629}]\label{LVA}
Let $G$ be a connected graph with maximum degree $\Delta$. If it is neither a complete graph of odd order nor a cycle, then the linear list vertex arboricity is at most $\lceil\frac{\Delta}{2} \rceil$. 
\end{theorem}
Indeed, we show the following stronger result, which together with \autoref{SMR} implies \autoref{LVA}. 
\begin{theorem}
Let $G$ be a connected graph with maximum degree $\Delta \geq 3$. Let $H$ be a cover of $G$ and $f$ be a function from $V(H)$ to $\{0, 1, \dots, \Delta -1\}$. If $(H, f)$ is not a building cover with $G = K_{\Delta + 1}$, and $f(v, 1) + f(v, 2) + \dots + f(v, s) \geq \Delta$, then $H$ has a strictly $f$-degenerate transversal. 
\end{theorem}

\begin{proof}
Suppose that $H$ has no strictly $f$-degenerate transversals. Note that $f(v, 1) + f(v, 2) + \dots + f(v, s) \geq \Delta \geq \deg_{G}(v)$ for each $v \in V(G)$. By \autoref{MR} and \autoref{FC}, the cover $(H, f)$ must be constructible and $f(v, 1) + f(v, 2) + \dots + f(v, s) = \deg_{G}(v) = \Delta$ for each $v \in V(G)$. This implies that $G$ is $\Delta$-regular and $f(v, 1) + f(v, 2) + \dots + f(v, s) = \Delta$ for each $v \in V(G)$. Since $G$ is $\Delta$-regular but the maximum value of $f$ is at most $\Delta - 1 $, each end-block must be a complete graph. Since $(H, f)$ is constructible, the regularity of $G$ implies that $G$ is the complete graph $K_{\Delta + 1}$ and $(H, f)$ should be a building cover of type (ii), which is a contradiction. 
\end{proof}

\section{Open discussions}
In this section, we give some problems that we are considering, and we present some challenging open problems. 
\subsection{Problem 1}
Farzad \cite{MR2538645} showed that every planar graph without $7$-cycles is $4$-choosable. Huang, Shiu and Wang \cite{MR2926103} showed that every planar graph without $7$-cycles has vertex arboricity at most two. 

\begin{conjecture}
Let $G$ be a planar graph without $7$-cycles. Let $H$ be a cover of $G$ and $f$ be a function from $V(H)$ to $\{0, 1, 2\}$. If $f(v, 1) + f(v, 2) + \dots + f(v, s) \geq 4$ for each $v \in V(G)$, then $H$ has a strictly $f$-degenerate transversal. 
\end{conjecture}

Recently, Kim \etal proved the following result which partially support the conjecture. 
\begin{theorem}[Kim, Liu and Yu, \cite{MR4108299}]
Every planar graph without $7$-cycles and butterflies is DP-4-colorable, where a butterfly is a graph isomorphic to the configuration depicted in \autoref{butterfly}. 
\end{theorem}

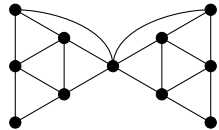
\begin{figure}[htbp]%
\centering
\begin{tikzpicture}
\coordinate (O) at (0, 0);
\coordinate (A1) at (150:1.5);
\coordinate (A2) at (210:1.5);
\coordinate (A3) at (-30:1.5);
\coordinate (A4) at (30:1.5);
\coordinate (B1) at ($0.5*(A1)$);
\coordinate (B2) at ($0.5*(A2)$);
\coordinate (B3) at ($0.5*(A3)$);
\coordinate (B4) at ($0.5*(A4)$);
\coordinate (C1) at ($0.5*($(A1)+(A2)$)$);
\coordinate (C2) at ($0.5*($(A3)+(A4)$)$);
\draw (O)--(A1)--(A2)--cycle;
\draw (O)--(A3)--(A4)--cycle;
\draw (B1)--(B2)--(C1)--cycle;
\draw (B3)--(B4)--(C2)--cycle;
\draw (O)..controls +(up: 0.5) and +(right:0.5)..(A1);
\draw (O)..controls +(up: 0.5) and +(left:0.5)..(A4);
\node[circle, inner sep = 1.5, fill, draw] () at (O) {};
\node[circle, inner sep = 1.5, fill, draw] () at (A1) {};
\node[circle, inner sep = 1.5, fill, draw] () at (A2) {};
\node[circle, inner sep = 1.5, fill, draw] () at (A3) {};
\node[circle, inner sep = 1.5, fill, draw] () at (A4) {};
\node[circle, inner sep = 1.5, fill, draw] () at (B1) {};
\node[circle, inner sep = 1.5, fill, draw] () at (B2) {};
\node[circle, inner sep = 1.5, fill, draw] () at (B3) {};
\node[circle, inner sep = 1.5, fill, draw] () at (B4) {};
\node[circle, inner sep = 1.5, fill, draw] () at (C1) {};
\node[circle, inner sep = 1.5, fill, draw] () at (C2) {};
\end{tikzpicture}
\caption{A butterfly.}
\label{butterfly}
\end{figure}

\subsection{Problem 2}
Choi and Zhang \cite{MR3233411} proved that every toroidal graph without $4$-cycles has vertex arboricity at most two. Chen, Huang and Wang \cite{MR3508765} extended this result and showed that every toroidal graph without $4$-cycles adjacent to $3$-cycles has list vertex arboricity at most two.

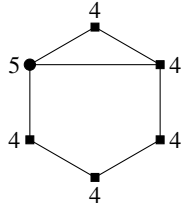
\begin{figure}[htbp]%
\centering
\begin{tikzpicture}
\coordinate (A) at (90:1);
\coordinate (B) at (150:1);
\coordinate (C) at (210:1);
\coordinate (D) at (270:1);
\coordinate (E) at (330:1);
\coordinate (F) at (30:1);
\draw (A) node[above]{\small $4$}--(B) node[left]{\small $5$}--(C) node[left]{\small $4$}--(D) node[below]{\small $4$}--(E) node[right]{\small $4$}--(F) node[right]{\small $4$}--cycle;
\draw (B)--(F);
\node[rectangle, inner sep = 1.5, fill, draw] () at (A) {};
\node[circle, inner sep = 1.5, fill, draw] () at (B) {};
\node[rectangle, inner sep = 1.5, fill, draw] () at (C) {};
\node[rectangle, inner sep = 1.5, fill, draw] () at (D) {};
\node[rectangle, inner sep = 1.5, fill, draw] () at (E) {};
\node[rectangle, inner sep = 1.5, fill, draw] () at (F) {};
\end{tikzpicture}
\caption{A reducible configuration for list vertex arboricity at most two.}
\label{F36}
\end{figure}
Chen, Huang and Wang \cite{MR3508765} showed that the configuration in \autoref{F36} is reducible for list vertex arboricity at most two. But Kim and Yu \cite{MR3969022} showed that this is irreducible for DP-$4$-coloring, so we cannot trivially obtain a solution to the corresponding DP-$4$-coloring problem. 

Very recently, Wang \etal \cite{Wang2019+} gave the following result which improves all the results mentioned in this subsection. 

\begin{figure}%
\centering
\subcaptionbox{\label{fig:subfig:a}}
{\begin{tikzpicture}
\coordinate (A) at (45:1);
\coordinate (B) at (135:1);
\coordinate (C) at (225:1);
\coordinate (D) at (-45:1);
\coordinate (H) at (90:1.414);
\draw (A)--(H)--(B)--(C)--(D)--cycle;
\draw (A)--(B);
\node[circle, inner sep = 1.5, fill, draw] () at (A) {};
\node[circle, inner sep = 1.5, fill, draw] () at (B) {};
\node[circle, inner sep = 1.5, fill, draw] () at (C) {};
\node[circle, inner sep = 1.5, fill, draw] () at (D) {};
\node[circle, inner sep = 1.5, fill, draw] () at (H) {};
\end{tikzpicture}}\hspace{1.5cm}
\subcaptionbox{\label{fig:subfig:b}}
{\begin{tikzpicture}
\coordinate (A) at (45:1);
\coordinate (B) at (135:1);
\coordinate (C) at (225:1);
\coordinate (D) at (-45:1);
\coordinate (H) at (90:1.414);
\coordinate (X) at ($(A)+(0, 1)$);
\coordinate (Y) at (45:2);
\draw (A)--(Y)--(X)--(H)--(B)--(C)--(D)--cycle;
\draw (X)--(A)--(H);
\node[circle, inner sep = 1.5, fill, draw] () at (A) {};
\node[circle, inner sep = 1.5, fill, draw] () at (B) {};
\node[circle, inner sep = 1.5, fill, draw] () at (C) {};
\node[circle, inner sep = 1.5, fill, draw] () at (D) {};
\node[circle, inner sep = 1.5, fill, draw] () at (H) {};
\node[circle, inner sep = 1.5, fill, draw] () at (X) {};
\node[circle, inner sep = 1.5, fill, draw] () at (Y) {};
\end{tikzpicture}}
\caption{Forbidden configurations in \autoref{IMRTHREE}.}
\label{A345}
\end{figure}
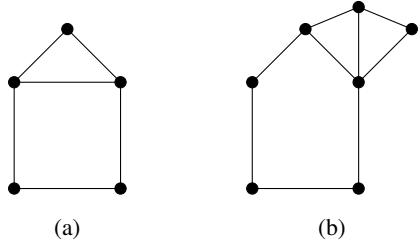

\begin{restatable}{theorem}{IMRTHREE}\label{IMRTHREE}
Let $G$ be a toroidal graph without subgraphs isomorphic to the configurations in \autoref{A345}. Let $H$ be a cover of $G$ and $f$ be a function from $V(H)$ to $\{0, 1, 2\}$. If $f(v, 1) + f(v, 2) + \dots + f(v, s) \geq 4$ for each $v \in V(G)$, then $H$ has a strictly $f$-degenerate transversal. 
\end{restatable}

Note that \autoref{fig:subfig:b} is a proper subgraph of \autoref{fig:subfig:b--} and \autoref{fig:subfig:c--}, thus the class of toroidal graphs without subgraphs isomorphic to the configurations in \autoref{E} is a superclass of that with respect to \autoref{A345}. 
\begin{conjecture}
Let $G$ be a toroidal graph without subgraphs isomorphic to the configurations in \autoref{E}. Let $H$ be a cover of $G$ and $f$ be a function from $V(H)$ to $\{0, 1, 2\}$. If $f(v, 1) + f(v, 2) + \dots + f(v, s) \geq 4$ for each $v \in V(G)$, then $H$ has a strictly $f$-degenerate transversal. 
\end{conjecture}
\subsection{Problem 3}
It is known that every triangle-free planar graph is $3$-colorable \cite[Gr\"{o}tzsch Theorem]{MR0116320}, but such graphs are not $3$-choosable \cite{MR1360130}. Note that every triangle-free planar graph is $3$-degenerate, thus such graphs are $4$-choosable. When triangles are allowed in planar graphs but the distance of triangles is at least one, the graphs are also $4$-choosable. Wang and Lih \cite{MR1935837} proved that every planar graph without intersecting triangles is $4$-choosable. Luo \cite{MR2292964} proved that every toroidal graph without intersecting triangles is $4$-choosable. Li and Wang proved the following result as a corollary in \cite[Corollary 1]{MR4294211}. 

\begin{theorem}[Li and Wang \cite{MR4294211}]
Every planar graph without intersecting triangles is DP-4-colorable. 
\end{theorem}

Chen \etal \cite{MR2889524} showed that every planar graph without intersecting triangles has vertex arboricity at most two. Analogously, Cai \etal \cite{MR3761240} proved that every planar graph without intersecting $5$-cycles has vertex arboricity at most two, while Li \etal \cite{Lv2019} proved that every planar graph without intersecting 5-cycles is DP-4-colorable. Very recently, Wang and Wang \cite{Wang2019+} showed the following theorem for planar graphs without intersecting 5-cycles. 

\begin{theorem}
Let $G$ be a planar graph without intersecting $5$-cycles. Let $H$ be a cover of $G$ and $f$ be a function from $V(H)$ to $\{0, 1, 2\}$. If $f(v, 1) + f(v, 2) + \dots + f(v, s) \geq 4$ for each $v \in V(G)$, then $H$ has a strictly $f$-degenerate transversal. 
\end{theorem}

So it is reasonable to give the following two conjectures. 

\begin{conjecture}
Let $G$ be a planar/toroidal graph without intersecting triangles. Let $H$ be a cover of $G$ and $f$ be a function from $V(H)$ to $\{0, 1, 2\}$. If $f(v, 1) + f(v, 2) + \dots + f(v, s) \geq 4$ for each $v \in V(G)$, then $H$ has a strictly $f$-degenerate transversal. 
\end{conjecture}

\begin{conjecture}
Let $G$ be a toroidal graph without intersecting $5$-cycles. Let $H$ be a cover of $G$ and $f$ be a function from $V(H)$ to $\{0, 1, 2\}$. If $f(v, 1) + f(v, 2) + \dots + f(v, s) \geq 4$ for each $v \in V(G)$, then $H$ has a strictly $f$-degenerate transversal. 
\end{conjecture}

\subsection{Problem 4}
Thomassen considered the vertex-partition of planar graphs into an $\ell_{1}$-degenerate graph and an $\ell_{2}$-degenerate graph. 
\begin{theorem}[Thomassen \cite{MR1358992}]
Every planar graph can be partitioned into a forest and a $2$-degenerate graph.
\end{theorem}

\begin{theorem}[Thomassen \cite{MR1866722}]
Every planar graph can be partitioned into an independent set and a $3$-degenerate graph.
\end{theorem}

As a consequence, we obtain the following corollary. 
\begin{corollary}
Every planar graph can be partitioned into a strictly $\ell_{1}$-degenerate graph and a strictly $\ell_{2}$-degenerate graph whenever $\ell_{1} + \ell_{2} = 5$ and $\ell_{1}, \ell_{2}$ are positive integers. 
\end{corollary}

Here, we pose the following conjecture which extends the above corollary if the conjecture is true. 
\begin{conjecture}
Let $G$ be a planar graph and $(H, f)$ be a positive-valued cover. If $s \geq 2$ and $f(v, 1) + \dots + f(v, s) \geq 5$ for each $v \in V(G)$, then $H$ has a strictly $f$-degenerate transversal. 
\end{conjecture}

Recently, Nakprasit and Nakprasit \cite{MR4114324} proved the following result which partially supports the conjecture. 
\begin{theorem}
Let $G$ be a planar graph. Let $H$ be a cover of $G$ and $f$ be a function from $V(H)$ to $\{0, 1, 2\}$. If $f(v, 1) + f(v, 2) + \dots + f(v, s) \geq 5$ for each $v \in V(G)$, then $H$ has a strictly $f$-degenerate transversal. 
\end{theorem}

Li and Wang \cite{Li} further extend it to the class of $K_{5}$-minor-free or $K_{3, 3}$-minor-free graphs. 

\begin{theorem}
Let $G$ be a $K_{5}$-minor-free or $K_{3, 3}$-minor-free graph. Let $H$ be a cover of $G$ and $f$ be a function from $V(H)$ to $\{0, 1, 2\}$. If $f(v, 1) + f(v, 2) + \dots + f(v, s) \geq 5$ for each $v \in V(G)$, then $H$ has a strictly $f$-degenerate transversal. 
\end{theorem}

\vskip 0mm \vspace{0.3cm} \noindent\textbf{Acknowledgments.} Many thanks to Sittitrai and Nakprasit for bringing \cite{MR1743629} to our attention. This work was supported by the Fundamental Research Funds for Universities in Henan (YQPY20140051). The authors would like to thank the anonymous referees for their valuable comments.

\end{document}